\documentclass[12pt]{amsart}
\usepackage{amsmath,amssymb,ifthen}
\usepackage{mathtools}
\usepackage[utf8]{inputenc}
\usepackage[T1]{fontenc}
\usepackage{amsthm}
\def\N{\mathbb{N}}
\def\R{\mathbb{R}}
\def\Q{\mathbb{Q}}
\def\IR{\mathbb{R}}
\def\IQ{\mathbb{Q}}

\def\Z{\mathbb{Z}}
\def\c{\mathfrak{c}}
\newcommand{\LIF}{\operatorname{LIF}}

\DeclareMathOperator{\LIN}{LIN}
\DeclareMathOperator{\dom}{dom}
\DeclareMathOperator{\rng}{rng}
\DeclareMathOperator{\id}{id}

\newtheorem{proposition}{Proposition}
\newtheorem{corollary}{Corollary}
\newtheorem{theorem}{Theorem}
\newtheorem{question}{Question}

\newcommand{\ra}{\rightarrow}

\newcommand{\on}{\operatorname}

\author[M.~Lichman]{Mateusz Lichman}
\address[M.~Lichman]{Institute of Mathematics, {\L}\'od\'z University of Technology, al. Politechniki 8, 93-590 {\L}\'od\'z, Poland}
\email{mateusz.lichman@wp.pl}

\author[M.~Pawlikowski]{Micha{\l} Pawlikowski}
\address[M.~Pawlikowski]{Institute of Mathematics, {\L}\'od\'z University of Technology, al. Politechniki 8, 93-590 {\L}\'od\'z, Poland}
\email{michal-pawlikowski4@wp.pl}

\author[Sz.~Smolarek]{Szymon Smolarek}
\address[Sz.~Smolarek]{Institute of Mathematics, {\L}\'od\'z University of Technology, al. Politechniki 8, 93-590 {\L}\'od\'z, Poland}
\email{szymon.smolarek123@wp.pl}

\author[J.~Swaczyna]{Jaros{\l}aw Swaczyna}
\address[J.~Swaczyna]{Institute of Mathematics, {\L}\'od\'z University of Technology, al. Politechniki 8, 93-590 {\L}\'od\'z, Poland}
\email{jaroslaw.swaczyna@p.lodz.pl}

\thanks{The last-named author acknowledge with thanks support received from {\L}\'od\'z University of Technology via "FU$^2$N - a fund for the improvement of the skills of young scientists".}

\title[Free group of Hamel functions]{Free group of Hamel functions}

\begin{document}

\maketitle

\begin{abstract}
    We construct a free group of continuum many generators among those autobijections of $\IR$ which are also Hamel bases of $\IR^2$, with identity function included. We also observe two new cases when a real function is a composition of two real functions which are Hamel bases of $\IR^2$.
\end{abstract}
\section{Introduction}
Hamel bases are very basic algebraic notion, firstly used to obtain non-continuous solutions to Cauchy's functional equation. Our note concerns some special Hamel bases of the real plane $\IR^2$, namely those bases which are simultaneously real functions (equivalently, have one-point intersection with every vertical line). Such sets were considered by other authors in the context of their measurability \cite{FNS, Nat, PR}, linear structure of the space $\IR^\IR$ \cite{P2, P3, P4}, lattice structure of $\IR^\IR$ \cite{Mat} and their relations to other families of real functions \cite{P1}.

For the present paper, article \cite{MN} is of particular importance. Its Authors constructed in \cite[Theorem 2.1]{MN} a bijection $f\colon \IR \to \IR$ which is a Hamel basis of the plane. The Authors also noted that \cite[Proof of Corollary 2.1]{MN} an inverse of a Hamel bijection is also a Hamel bijection and \cite[Corollary 2.3]{MN} each real function is a composition of three Hamel functions. Those results were the motivations for our research on behaviour of Hamel functions in the context of function composition. The aim of this note is to prove that there exists a free group of $\mathfrak{c}$ many generators within the family of all Hamel autobijetions of $\IR$.

\section{Preliminaries}

For any set $X$ by $|X|$ we denote its cardinality. By Hamel basis we mean a set $B\subset \R^2$ which is a basis of the linear space $\IR^2$ over $\IQ$. We say that $f\colon \R \to \IR$ is a Hamel function ($f\in \on{HF}$ in short) if $f$ treated as a subset of $\R^2$ is a Hamel basis. Now we recall the definition of free group.

Let $(G, \circ)$ be a group. For a given $g\in G$ and $n\in \N$ we will denote $g\circ \ldots  \circ g$ by $g^n$ and the inverse element of $g$ by $g^{-1}$. For any $z\in \Z, z<0$ we define $g^z\coloneqq (g^{-1})^{-z}$, $g^0=e$ where $e$ is the neutral element of group $(G, \circ)$. For any $z\in \Z, g^z$ will be called the $z$-th power of $g$. We will say that $G$ is a \textit{free group} if there exists a subset $S\subset G$ such that for every $g\in G$ there exists a unique (disregarding trivial variations such as $a\circ b=a\circ c\circ c^{-1}\circ b$ or $a^5=a^2\circ a^3$) representation of $g$ as a product of finitely many powers of elements of $S$. Then the set $S$ will be called the \textit{set of generators}, its elements - generators and the elements of $G$ - words.

For sets $X, Y$, $Z \subset X$ and $f\colon  Z \to Y$ we say that $f$ is a partial function from $X$ to $Y$ and write $f:X\rightharpoonup Y$. In such a case we denote domain of $f$ by $\on{dom}(f)$ and its range by $\on{rng}(f)$.

We say that a function $f\colon X\rightharpoonup \R$, where $X\subset \IR$, is a \emph{partially linearly independent function}  ($f\in \on{PLIF}$ in short) if it is linearly independent over $\Q$ as a subset of $\IR^2$. If moreover $X=\IR$, then we simply say $f$ is a \emph{linearly independent function}  ($f\in \on{LIF}$ in short). For $A\subset \R^2$ by $\on{LIN}_{\Q}(A)$ we will denote the linear subspace of $\IR^2$ over $\Q$ generated by $A$.

The following Proposition is very useful for working with Hamel functions.
\begin{proposition}\label{T:dost HF} \cite[Fact 2.2.]{MN} Let $f:\R\ra \R$ be a function, $f\in \LIF$. Then $f\in \on{HF}$ if and only if $\langle0, x\rangle\in \on{LIN}_{\Q}(f)$ for every $x\in \IR$.
\end{proposition}

\section{Main result}
Our goal is to prove the following.
\begin{theorem}\label{T: main}
There exists a family $\{f_\beta: \beta < \c \}$ of Hamel bijections such that they are free generators of a group with respect to the composition and all elements of this group but identity are also Hamel bijections.
\end{theorem}
\begin{proof}
    
The general idea of the proof is to construct functions $f_\beta$, for $\beta< \c$, inductively. Those functions will be free generators of the desired group. By ${_\kappa} f_\beta$ we will denote the state of function $f_\beta$ at $\kappa$ stage of our construction.
We define
$$A\coloneqq \bigcup_{m\geq 1}(\Z\setminus \{0\})^m\times \c^m,$$
$$B\coloneqq \{(n_0, \ldots  , n_{m-1}, \gamma_0, \ldots  , \gamma_{m-1})\in A: \exists i<m-1 \ \ \gamma_i=\gamma_{i+1}, \ \ m \geq 1\}$$
and enumerate
$$A\setminus B=\{(n_0^{\alpha}, \ldots  , n_{m_{\alpha}-1}^{\alpha}, \gamma_0^{\alpha}, \ldots  , \gamma_{m_{\alpha}-1}^{\alpha}): \alpha< \c\}.$$
The set $A\setminus B$ is identified with all possible words written in a reduced form.

Let $\R\times \c=\{(x_{\kappa}, \alpha_{\kappa}): \kappa< \c \}$
be a well-ordering of $\R \times \c$.

For $\alpha, \kappa<\c$, by $_{\kappa}h_{\alpha}$ we denote the word
$$_{\kappa}f_{\gamma_{m_{\alpha}-1}^\alpha}^{n_{m_{\alpha}-1}^{\alpha}} \circ \ldots  \circ _{\kappa}f_{\gamma_0^{\alpha}}^{n_0^{\alpha}}.
$$
Similarly by $h_\alpha$ we denote the word 
$$
h_\alpha=f_{\gamma_{m_{\alpha}-1} ^ \alpha}^{n_{m_{\alpha}-1}^{\alpha}} \circ \ldots  \circ f_{\gamma_0^{\alpha}}^{n_0^{\alpha}}=\bigcup_{\kappa < \c} \  _{\kappa}f_{\gamma_{m_\alpha -1}^\alpha}^{n_{m_{\alpha}-1}^\alpha}\circ \ldots  \circ \  _{\kappa}f_{\gamma_0^{\alpha}}^{n_0^{\alpha}}=\bigcup_{\kappa<\c}\ _{\kappa}h_{\alpha}
$$
More precisely, for $\beta<\c$ we construct a sequence of partial functions $_{\kappa}f_{\beta}$, $\kappa<\c$ such that the following conditions hold.
\begin{itemize}
\item[(I)]
$_{\kappa}f_{\beta_n}^{m_n} \circ \ldots  \circ _{\kappa}f_{\beta_1}^{m_1}\in \on{PLIF}$ for all $n\in \N, m_1, \ldots  , m_n \in \Z\setminus\{0\}, \beta_1\neq \beta_2 \neq\ldots  \neq \beta_n \in \c$. Note that $\beta$'s need not to be pairwise different.
\item[(II)]
$_{\kappa}f_{\beta}$ is one-to-one.

\item[(III)]$ |\bigcup_{\beta < \c}      
{}_{\kappa}f_{\beta}|\leq \omega+|\kappa|$.

\item[(IV)]
$_{\kappa}f_{\beta} \subset \  _{\gamma}f_{\beta}$ for $\kappa<\gamma$.
\item[(V)]
$\langle 0, x_{\kappa}\rangle \in \LIN_{\Q}(_{\kappa+1}h_{\alpha_{\kappa}})$.
\item[(VI)]
$x_{\kappa}\in \dom(_{\kappa+1}f_{\alpha_{\kappa}})$.
\item[(VII)]
$x_{\kappa}\in \rng(_{\kappa+1}f_{\alpha_{\kappa}})$.
\end{itemize}
Then for $\beta < \c$ let $f_{\beta} \coloneqq \bigcup_{\kappa<\c}\ _{\kappa}f_{\beta}$. 
In condition (I) we consider the composition just at those points at which it is already defined.

Note that once we are done, conditions (II), (VI) and (VII) guarantee that for every $\beta < \c$, $f_{\beta}$ is an autobijection of $\R$. Therefore every word constructed from the set $\{f_{\beta}: \beta < \c\}$ is an autobijection. $\circ$ is a group operation on the set of all words.

The condition (I) assures that every word is LIF. Using Proposition \ref{T:dost HF} and the condition (V) we get that every word is a Hamel function.

Observe now, that condition (I) also assures that the set $\{f_{\beta}: \beta < \c\}$ generates a free group.

Indeed, let $n, m\in \N$, $k_1, \ldots  , k_n, l_1, \ldots  , l_m\in \Z$, $\beta_1 \neq \ldots  \neq \beta_n, \gamma_1 \neq \ldots  \neq \gamma_m \in \c$ and suppose that
$$g\coloneqq f_{\beta_n}^{k_n} \circ \ldots  \circ f_{\beta_1}^{k_1}=f_{\gamma_m}^{l_m} \circ \ldots  \circ f_{\gamma_1}^{l_1}\eqqcolon h$$
but
\begin{equation}\label{rown1}
(k_1, \ldots  , k_n, \beta_1, \ldots  , \beta_n)\neq (l_1, \ldots  , l_m, \gamma_1, \ldots  , \gamma_m)
\end{equation}
i.e., the representation of the word $g$ is not unique.

Since $g=h$ we get that $g\circ h^{-1}=\id_{\IR}$. However  \eqref{rown1} implies that word $g\circ h^{-1}$ written in reduced form remains nontrivial. Thus there exists $p\leq m+n$, $r_1, \ldots r_p\in \Z\setminus\{0\}$, $\delta_1 \neq \ldots  \neq \delta_p\in \c$ such that $\id_{\IR}=f_{\delta_p}^{r_p}\circ \ldots  \circ f_{\delta_1}^{r_1}$, which contradicts $f_{\delta_p}^{r_p}\circ \ldots  \circ f_{\delta_1}^{r_1}$ being $\on{LIF}$, as $\on{id}_{\IR}\notin \on{LIF}$. 

As a consequence we get that for $\alpha, \beta <\c$, $\alpha\neq \beta$, $f_{\alpha}\neq f_{\beta}$ and thus the cardinality of $\{f_{\beta}: \beta < \c\}$ is $\c$.

To sum up, we showed that $H=\{h_{\alpha}: \alpha<\c\}\cup\{\id\}$ equipped with the action of functions composition is the desired free group of $\c$ many generators.

It remains to construct functions $_\kappa f_\alpha$ satisfying the conditions (I)-(VII), so take $\gamma<\c$. Assume that for all $\beta<\c$ partial functions $_{\alpha}f_{\beta}$ are constructed and conditions (I)-(VII) hold for $\alpha<\gamma$. If $\gamma=\emptyset$, then we let $_\gamma f_\beta\coloneqq \emptyset$ for $\beta<\c$. If $\gamma > 0$ is a limit ordinal, then for $\beta<\c$ set $_\gamma f_\beta\coloneqq \bigcup_{\alpha<\gamma}$ $_\alpha f_\beta$. Otherwise there is $\kappa$ with $\kappa+1=\gamma$.
\begin{center}
\textbf{STEP I}
\end{center}

In this step we assure that condition (V) is true and conditions (I)-(IV) hold. If $\langle 0, x_{\kappa}\rangle \in \on{LIN}_{\Q}(_{\kappa}h_{\alpha_{\kappa}})$ then set
$$_{\kappa}f'_{\beta}\coloneqq _{\kappa}f_{\beta}$$
for all $\beta<\c$. Otherwise the general idea will be to find $x,y \in \IR$ and extend existing functions in such a way that $_\kappa h_\alpha(x) = y$ and $_\kappa h_\alpha(-x) = x_\kappa - y$ (in order to simplify the notation we set $\alpha:=\alpha_\kappa$). Once it will be done, condition (V) will be satisfied. However, we must make sure that conditions (I)-(IV) still hold, so let us set 
$$C\coloneqq \{x_{\kappa}\}\cup \bigcup_{\alpha<\c} (\dom(_{\kappa}h_{\alpha})\cup \rng(_{\kappa}h_{\alpha}))=\{x_{\kappa}\}\cup \bigcup_{\beta<\c} (\dom(_{\kappa}f_{\beta})\cup \rng(_{\kappa}f_{\beta})).$$

Note that at the step $\kappa$ at most $|\kappa|+\omega$ of words $_\kappa h_\alpha$ are nonempty functions and cardinality of each $_\kappa h_\alpha $ is at most $|\kappa|+\omega$, thus the set $C$ has cardinality less than continuum.

Recall that $h_\alpha=\bigcup_{\kappa < \c} \  _{\kappa}f_{\gamma_{m_\alpha -1}^\alpha}^{n_{m_{\alpha}-1}^\alpha}\circ \ldots  \circ \  _{\kappa}f_{\gamma_0^{\alpha}}^{n_0^{\alpha}}$, let $s\coloneqq \sum\limits_{i=0}^{m_{\alpha}-1} |n^{\alpha}_i |$ be the number of letters used in the word $h_\alpha$, and choose (one can do it since $|C|<\c$ so $C$ does not span $\R$) $x\in \R\setminus \on{LIN}_{\Q}(C)$. Let $z_0\coloneqq x$ and $z_{l+1}\in \R\setminus \on{LIN}_{\Q}(C\cup\{z_0, \ldots  , z_l\})$ for $l<s-1$.

Set
$$D\coloneqq C\cup \{z_0, \ldots  , z_{s-1}\},$$
 $r_0\coloneqq -x$ and pick $r_{l+1}\in \R\setminus \on{LIN}_\Q(D\cup \{r_0, \ldots  , r_l\})$ for $l<s-1$. Finally choose $y\in \R\setminus \on{LIN}_\Q(D\cup\{r_0, \ldots  r_{s-1}\})$ and let $y'\coloneqq x
_\kappa-y$. Then also $y'\in \R\setminus \on{LIN}_\Q(D\cup\{r_0, \ldots  r_{s-1}\})$
\\Let $z_s\coloneqq y$ and $r_s\coloneqq y'$.
\\Our goal now is to make functions $\ {_\kappa}{f'}_{\gamma_{m_{\alpha}-1}^{\alpha}}^{n_{m_{\alpha}-1}^{\alpha}}, \ldots, {_\kappa}{f'}_{\gamma_0^{\alpha}}^{n_0^{\alpha}}$ such that for
\begin{center}
$_{\kappa}h'_{\alpha}\coloneqq 
\ {_\kappa}{f'}_{\gamma_{m_{\alpha}-1}^{\alpha}}^{n_{m_{\alpha}-1}^{\alpha}} \circ \ldots  \circ {_\kappa}{f'}_{\gamma_0^{\alpha}}^{n_0^{\alpha}}$
\end{center}
 we have $_{\kappa}h'_{\alpha}(x)=y$ and $_{\kappa}h'_{\alpha}(-x)=y'$. Therefore we define
\begin{itemize}
\item $p_0\coloneqq 0$
\item $p_j\coloneqq \sum\limits_{i=0}^{j-1}|n_i^{\alpha}|$ for $0<j<m_{\alpha}$
\item $p_{m_{\alpha}}\coloneqq s$
\end{itemize} 
and then for $j< m_{\alpha_\kappa}$:
\begin{itemize}
\item[(a)] if $n_j^{\alpha}<0$ then we set\\
${_\kappa}{f'}_{\gamma_j^{\alpha}}\coloneqq {_\kappa}f_{\gamma_j^{\alpha}}\cup \{\langle z_{p_j+1}, z_{p_j}\rangle, \ldots  , \langle z_{p_{j+1}}, z_{p_{j+1}-1}\rangle, \langle r_{p_j+1}, r_{p_j}\rangle, \ldots  , \langle r_{p_{j+1}}, r_{p_{j+1}-1}\rangle\} $
\item[(b)] $n_j^{\alpha}>0$ then we set\\
$_{\kappa}f'$ $_{\gamma_j^{\alpha}}\coloneqq {_\kappa}f$ $_{\gamma_j^{\alpha}}\cup \{\langle z_{p_j}, z_{p_j+1}\rangle, \ldots  , \langle z_{p_{j+1}-1}, z_{p_{j+1}}\rangle, \langle r_{p_j}, r_{p_j+1}\rangle, \ldots  , \langle r_{p_{j+1}-1}, r_{p_{j+1}}\rangle\} $
\end{itemize}
The idea behind above enlarging is to ensure that ${_\kappa} f_{\gamma^\alpha_j}^{n_j^\alpha}(z_{p_j})=z_{p_{j+1}}$ and ${_\kappa} f_{\gamma^\alpha_j}^{n_j^\alpha}(r_{p_j})=r_{p_{j+1}}$, so roughly speaking to go from $x$ to $y$ via $z$'s and from $-x$ to $y'$ through $r$'s.

In fact, we are abusing the notation a bit, just to avoid introducing more indices. One generator might occur in the word $_{\kappa}h_{\alpha}$ several times. If $f$ is such a partial function and we have already added some points to it, then we keep enlarging the enlarged function $f$ instead of coming back to its previous form.

Then our goal is reached and therefore $\langle 0, x_\kappa\rangle \in \on{LIN}_\Q(_{\kappa} h'_{\alpha})$ since $\langle x, y\rangle, \langle -x, y'\rangle \in {_\kappa} h'_{\alpha}$ and $x_\kappa=y'+y$.

For those functions $_\kappa f_\beta$ that we have not changed during Step I we let $_\kappa f'_\beta\coloneqq _\kappa f_\beta$.
\\Now we will check that conditions (I)-(IV) still hold. Conditions (II)-(IV) are trivially seen to remain true. Note that all but those functions that occur in the word $_\kappa h_{\alpha}$ remain unchanged. This and the definition of $C$ implies that the only words that may have changed during Step I are those words that are composed of partial functions $_\kappa f'_{\gamma_j^{\alpha}}$, $j<m_{\alpha}$.

Let $h$ be such a partial function. Let $E$ be the set of points that we enlarged function $h$ by in Step I. If $E=\emptyset$, we have nothing to check. Assume that $E\neq \emptyset$. Let $k, n \in \N$, $q_1, \ldots  , q_{k+n} \in \Q$ and 
$$\langle u_1, y_1 \rangle, \ldots  , \langle u_k, y_k \rangle \in h\setminus E$$
$$\langle u_{k+1}, y_{k+1} \rangle, \ldots  , \langle u_{k+n}, y_{k+n} \rangle \in E$$
where $x_i\neq x_j$ for $i\neq j$. Suppose that
$$\sum\limits_{i=1}^{k+n} q_i\langle u_i, y_i \rangle = \langle 0, 0 \rangle$$
Then
$$\sum\limits_{i=1}^{k+n} q_iu_i = 0$$

We consider two cases.
\begin{itemize}
\item[(1)] If $\{x, -x\}\not\subset \{u_{k+1}, \ldots  , u_{k+n}\}$, then for each $i\leq n$ we have 
$$x_{k+i}\not\in \on{LIN}_\Q(\{u_1, \ldots  , u_{k+n}\}\setminus \{u_{k+i}\})$$ (see the definition of $z_0, \ldots  , z_s, r_0, \ldots  , r_s$), thus $q_{k+i}=0$. We get that
$$\sum\limits_{i=1}^k q_i\langle u_i, y_i \rangle = \langle 0, 0 \rangle$$
and since $h\setminus E \in \on{PLIF}$, we get that $q_i=0$ for $i\leq k$. 
\item[(2)] If $\{x, -x\}\subset \{u_{k+1}, \ldots  , u_{k+n}\}$, then $x=u_j$, $-x=u_l$ for some 
$$j, l \in \{k+1, \ldots  , k+n \}.$$
Then for each $i\leq n$, $l-k\neq i\neq j-k$ we have 
$$u_{k+i}\not\in \on{LIN}_\Q(\{u_{k+1}, \ldots  , u_{k+n}\}\setminus \{u_{k+i}\})$$
and like in case (1) we get $q_{k+i}=0$ for those $i$. Hence we get
$$\sum\limits_{i=1}^{k} q_i\langle u_i, y_i \rangle + q_j\langle x, y_j \rangle + q_l\langle -x, y_l \rangle = \langle 0, 0 \rangle$$
and
$$
\sum\limits_{i=1}^{k} q_iy_i + q_jy_j + q_ly_l = 0.
$$
Since $y_l\not\in \on{LIN}_\Q(\{y_1, \ldots  , y_k, y_j \})$ we have $q_l=0$. Similarly we show that $q_j=0$ and thus we get
$$\sum\limits_{i=1}^k q_i\langle u_i, y_i \rangle = \langle 0, 0 \rangle$$
and since $h\setminus E \in \on{PLIF}$ we get that $q_i=0$ for $i\leq k$. 
\end{itemize}

Finally we get that $h\in \on{PLIF}$.

\begin{center}
\textbf{STEP II}
\end{center}

In this step we assure that condition (VI) is true and conditions (I)-(V) hold. 

If $x_{\kappa}\in \dom(_{\kappa}f'_{\alpha_{\kappa}})$ then set $_\kappa f''_{\beta}\coloneqq {} _\kappa f'_{\beta}$ for all $\beta<\c$. Otherwise let 
$$F\coloneqq \bigcup_{\alpha<\c} (\dom(_{\kappa}h'_{\alpha})\cup \rng(_{\kappa}h'_{\alpha}))=\bigcup_{\beta<\c} (\dom({_\kappa}f'_{\beta})\cup \rng({_\kappa}f'_{\beta}))$$

Note that at the step $\kappa$ at most $|\kappa|+\omega$ of words $_\kappa h'_\alpha$ are nonempty functions and cardinality of each $_\kappa h'_\alpha $ is at most $|\kappa|+\omega$, thus set $F$ has cardinality less than continuum. Therefore one can choose $y\in \R\setminus \LIN_\Q(F)$. Then define $_\kappa f''_{\alpha_{\kappa}}\coloneqq {_\kappa} f'_{\alpha_{\kappa}} \cup \{\langle x_\kappa, y \rangle \}$. For $\beta \neq \alpha_\kappa$ let $_\kappa f''_\beta\coloneqq {_\kappa} f'_\beta$. 

In any case by $_\kappa h''_\alpha$ we denote suitable composition of $_\kappa f''_\beta$'s.

Conditions (II)-(VI) clearly hold. We will check condition (I). Note that we only need to consider words that were enlarged by the point $\langle x, y \rangle$ for some $x\in \R$. Let h be such a word. Let $k \in \N$, $q_1, \ldots  , q_k, p \in \Q$ and 
$$\langle u_1, y_1 \rangle, \ldots  , \langle u_k, y_k \rangle, \langle x, y \rangle \in h$$
where $x\neq u_i\neq u_j$ for $i\neq j$. Suppose that
$$\sum\limits_{i=1}^{k} q_i\langle u_i, y_i \rangle + p\langle x, y \rangle= \langle 0, 0 \rangle$$
Then
$$\sum\limits_{i=1}^{k} q_iy_i + py = 0$$
Since $y\not\in \on{LIN}_\Q(\{y_1, \ldots  , y_k\})$ we get that $p=0$ and
$$\sum\limits_{i=1}^{k} q_i\langle u_i, y_i \rangle = \langle 0, 0 \rangle$$
Since $h\setminus\{\langle x, y \rangle\}\in \on{PLIF}$ we get that $q_1= \ldots  =q_k=0$.

Finally $h\in \on{PLIF}$.

\begin{center}
\textbf{STEP III}
\end{center}

In this step we assure that condition (VII) is true and conditions (I)-(VI) hold.

If $x_{\kappa}\in \on{rng}(_{\kappa}f''_{\alpha_{\kappa}})$ then set $_{\kappa+1} f_{\beta}\coloneqq _\kappa f''_{\beta}$ for $\beta<\c$.
Otherwise let 
$$G\coloneqq \bigcup_{\alpha<\c} (\on{dom}(_{\kappa}h''_{\alpha})\cup \on{rng}(_{\kappa}h''_{\alpha}))=\bigcup_{\beta<\c} (\on{dom}(_{\kappa}f''_{\beta})\cup \on{rng}(_{\kappa}f''_{\beta}))$$

Note that at the step $\kappa$ at most $|\kappa|+\omega$ of words $_\kappa h''_\alpha$ are nonempty functions and cardinality of each $_\kappa h''_\alpha $ is at most $|\kappa|+\omega$, thus the set $G$ has cardinality less than continuum. Therefore one can choose $x\in \R\setminus \on{LIN}_\Q(G)$. Then define $_{\kappa+1}f_{\alpha_{\kappa}}\coloneqq {_\kappa} f''_{\alpha_{\kappa}} \cup \{\langle x, x_\kappa \rangle \}$. For $\beta \neq \alpha_\kappa$ let $_{\kappa+1}f_\beta\coloneqq {_\kappa} f''_\beta$. 

Conditions (II)-(VII) clearly hold. We will check condition (I). Note that we only need to consider words that were enlarged by the point $\langle x, y \rangle$ for some $y\in \R$. Let h be such a word. Let $k \in \N$, $q_1, \ldots  , q_k, p \in \Q$ and 
$$\langle u_1, y_1 \rangle, \ldots  , \langle u_k, y_k \rangle, \langle x, y \rangle \in h$$
where $x\neq u_i\neq u_j$ for $i\neq j$. Suppose that
$$\sum\limits_{i=1}^{k} q_i\langle u_i, y_i \rangle + p\langle x, y \rangle= \langle 0, 0 \rangle.$$
Then
$$\sum\limits_{i=1}^{k} q_iu_i + px = 0.$$
Since $x\not\in \LIN_\Q(\{u_1, \ldots  , u_k\})$ we get that $p=0$ and
$$\sum\limits_{i=1}^{k} q_i\langle u_i, y_i \rangle = \langle 0, 0 \rangle.$$
Since $h\setminus\{\langle x, y \rangle\}\in \on{PLIF}$ we get that $q_1= \ldots  =q_k=0$.

Finally $h\in \on{PLIF}$.
\end{proof}
The above Theorem suggest the following question:
\begin{question}
Characterize these groups whose isomorphic copies may be found within the family of Hamel bijections with identity function included. In particular, is it possible to find there a free group of $2^\c$ many generators?
\end{question}
\section{Remarks}

We conclude our note by some observations concerning \cite[Problem 2.1]{MN}. As stated in \cite[Corollary 2.3]{MN}, each real function is a composition of three Hamel functions, and it is an open problem if there is a real function for which two Hamel functions won't be enough. By  \cite[Theorem 2.4]{MN} every $\on{LIF}$ function is a composition of two Hamel functions. We will observe that similar result holds under different assumptions.
\begin{corollary}\label{stale}
If $g$ is a constant function, then $g$ is a composition of two Hamel functions.
\end{corollary}
\begin{proof}
Let $g$ be a constant function $g\equiv c$. By \cite[Theorem 2.3]{MN}, there is a Hamel function $f$ whose range is linearly independent over $\mathbb{Q}$. Then let us set $h_0: \on{rng}(f)\to \IR$ to be constantly equal to $c$, and then let $h$ be a Hamel function extending the $h_0$ obtained by \cite[Lemma 2.1]{MN}. Note that $g=h \circ f$. 
\end{proof}

\begin{corollary}
If $g: \on\IR \to\IR$ and there exists a Hamel basis $H$ such that $g|_{H\cup \{0\}}$ is constant then $g$ is a composition of two Hamel functions.
\end{corollary}
\begin{proof}
Let $g: \on\IR \to\IR$ and $H$ be a Hamel basis. Assume that $g|_{H\cup \{0\}}$ is constant. Let $H'$ be any Hamel basis. Let $f|_{\IR\setminus H}$ be a bijection onto $H'$ and $f|_H \equiv f(0)$. By the proof of \cite[Theorem 2.3]{MN}, $f$ is a Hamel function. Clearly, $\on{rng}(f)=H'$. 
\\Now, let $h_0: \on H' \to \IR$ be defined by $h_0=g \circ f^{-1}$. There is one point in $H'$, namely $f(0)$, on which $f^{-1}$ is multi-valued. Nevertheless, the composition still makes sense as $g$ is constant on the respective set. By \cite[Lemma 2.1]{MN}, $h_0$ can be extended to a Hamel function $h: \on\IR \to\IR$. Clearly, $g=h\circ f$.
\end{proof}

\subsection*{Acknowledgements} We are indebted to Szymon Głąb for drawing our attention to groups consisting of Hamel autobijections of $\IR$.

\bibliographystyle{plain}
{}
\end{document}